\documentclass[A4,5p]{elsarticle}

\usepackage{multicol}
\usepackage{color}
\usepackage{xspace}
\usepackage{pdfwidgets}
\usepackage{enumerate}
\usepackage{url}

\usepackage[utf8]{inputenc}
\usepackage{comment}
\usepackage{amscd}
\usepackage{amsmath}
\usepackage{eucal}
\usepackage{todonotes}
\usepackage{amssymb}
\usepackage{pstricks}
\usepackage{graphicx}
\usepackage{lineno}
\usepackage{subcaption}
\usepackage{empheq}
\usepackage{amsthm}
\usepackage[all]{xy}
\usepackage{tikz-cd}
\usetikzlibrary{matrix}
\usepackage{tikz}

\newtheorem{definition}{Definition}
\newtheorem{theorem}{Theorem}
\DeclareMathOperator{\rank}{Rank}

\newtheorem{lemma}{Lemma}

\newtheorem{proposition}{Proposition}
\newtheorem{remark}{Remark}
\newtheorem{example}{Example}

\numberwithin{equation}{section}

\makeatletter
\def\bs{\expandafter\@gobble\string\\}
\def\lb{\expandafter\@gobble\string\{}
\def\rb{\expandafter\@gobble\string\}}
\def\@pdfauthor{C.V.Radhakrishnan}
\def\@pdftitle{elsarticle.cls -- A documentation}
\def\@pdfsubject{Document formatting with elsarticle.cls}
\def\@pdfkeywords{LaTeX, Elsevier Ltd, document class}

\DeclareRobustCommand{\LaTeX}{L\kern-.26em%
        {\sbox\z@ T%
         \vbox to\ht\z@{\hbox{\check@mathfonts
           \fontsize\sf@size\z@
           \math@fontsfalse\selectfont
          A\,}%
         \vss}%
        }%
     \kern-.15em%
    \TeX}
\makeatother

\begin{document}

\def\testa{This is a specimen document. }
\def\testc{\testa\testa\testa\testa}
\def\testb{\testc\testc\testc\testc\testc}
\long\def\test{\testb\par\testb\par\testb\par}

\pinclude{\copy\contbox\printSq{\LastPage}}

\title{Detectability via observability in a nonuniform framework: dual relationship with controllability and stabilizability}



\author[Huerta]{Ignacio Huerta}
\ead{ignacio.huertan@usm.cl}  
\author[Monzón]{Pablo Monzón}\ead{monzon@fing.edu.uy}  

\address[Huerta]{Departamento de Matem\'atica, Universidad T\'ecnica Federico Santa Mar\'ia, Casilla 110-V, Valpara\'iso, Chile.}  
\address[Monzón]{Facultad de Ingenier\'ia -- Universidad de la Rep\'ublica, Julio Herrera y Reissig 565, Montevideo, Uruguay.}

\begin{abstract}
    In this paper we propose a new observability property for nonautonomous linear control systems in finite dimension: the nonuniform complete observability, which is more general than the uniform complete observability. A dual relationship is established between this new notion of observability and the recently defined nonuniform complete controllability property, with the aim of obtaining the main result, which proves that nonuniform complete observability guarantees the nonuniform exponential detectability, a concept related to exponential stability in this framework.
\end{abstract}



\date{\today}
\maketitle


\section{Introduction}\label{sec-introduction}
In this work we deal with linear time-varying systems of the form
\begin{subequations}
	\begin{empheq}[left=\empheqlbrace]{align}
		& \dot{x}(t)=A(t)x(t)+B(t)u(t), \label{LTIsystem1a} \\
		& y(t)=C(t)x(t)+D(t)u(t), \label{LTIsystem1b}
	\end{empheq}
\end{subequations}
where $x(t)\in\mathbb{R}^{n}$ is the \textit{state} of the system, $u(t)\in\mathbb{R}^{p}$ is the \textit{input}, $y(t)\in\mathbb{R}^{m}$ is the \textit{output} and $A(t), B(t), C(t), D(t)$ are measurable, bounded on finite intervals, matrix valued functions from $\mathbb R$ to $M_{n\times n}(\mathbb R)$, $M_{n\times p}(\mathbb R)$, $M_{m\times n}(\mathbb R)$ and $M_{m\times p}(\mathbb R)$ respectively. On the one hand, we present new results on observability, i.e., the ability of the system to infer in finite time the actual or the initial state, through the observation of the output and the input, while on the other hand we present a result on detectability, which allows us to design an asymptotically estimator of the state. The article focus on nonuniform time-varying linear systems and extends the original results of Rudolph Kalman \cite{Kalman}, introducing a new class of systems: nonuniform completely observable systems. 

In a previous work we explore a new controllability property \cite{HMR} and we related it with the classical Kalman definitions and new results coming from the recent qualitative theory of linear nonautonomous differential equations: the bounded growth property and the dichotomies theory \cite{Palmer}. In that work, we were able to present the Nonuniform Complete Controllability property which is more general than the classical Uniform Complete Controllability and less restrictive than simple Complete Controllability \cite{Kalman}. Besides, we characterized this new concept in terms of Gramians and we presented some non-trivial examples. In this manuscript we extend those ideas to the observability problem, using the adjoint and dual systems of the original one.

We recall the following notation which will be used in this paper. The inner product of two vectors $x,y\in \mathbb{R}^{n}$ is denoted by
$\langle y, x\rangle=x^{T}y$ and the Euclidean norm of a vector will be denoted by $|x|=\sqrt{\langle x,x\rangle}$. A symmetric matrix $M=M^{T}\in M_{n}(\mathbb{R})$ is semi--positive definite if
$x^{T}Mx =\langle M x, x\rangle  \geq 0$ for any real vector $x\neq 0$, and this property will be denoted as $M \geq 0$. In case that the above inequalities are strict, we say that $M=M^{T}$ is positive definite. Given two matrices $M, N\in M_{n}(\mathbb{R})$, we write $M\leq N$ if  $N-M\geq0$. 

\subsection{Observability}
Observability analysis of \eqref{LTIsystem1a}-\eqref{LTIsystem1b} can be reduced to the study of the simpler system
\begin{subequations}
	\begin{empheq}[left=\empheqlbrace]{align}
		& \dot{x}(t)=A(t)x(t), \label{control1a} \\
		& y(t)=C(t)x(t), \label{control1b}
	\end{empheq}
\end{subequations}
and the mapping from the state to the output and its inverse, which allows the possibility of recovering the state from the output \cite{Callier}. We formalize some definitions. The system \eqref{control1a}-\eqref{control1b} is {\it observable at time $t_{0}$} if and only if there exists
$t_{f}>t_{0}$ such that $M(t_{0},t_{f})>0$. In addition, it is {\it completely observable} (CO) if and only if for 
any $t_{0}$, there exists $t_{f}>t_{0}$ such that $M(t_{0},t_{f})>0$. We refer the reader to \cite{Anderson67,Kalman,Kreindler} for a detailed description.

For linear time-invariant systems, observability is equivalent to the complete rank of the observability matrix \cite{Kalman}:
$$
\mathcal{O}=\left[\begin{array}{c}
	C \\
	C A \\
	C A^2 \\
	\vdots \\
	C A^{n-1}
\end{array}\right]\in M_{nm\times n}(\mathbb{R})
$$
and also to the invertibility of the observability Gramian
 $$
  M(t_0,t_1)=\displaystyle\int_{t_0}^{t_1}e^{A^{T}(s-t_0)}C^{T}Ce^{A(s-t_0)}\;ds,
  $$
for any value $t_1>t_0$. The former idea can not be directly extended to the time-varying case, but the latter leads to the invertibility of the general observability Gramian
\begin{equation}
	\label{gramianobservability}
	M(t_0,t_f)=\displaystyle\int_{t_0}^{t_f}\Phi_{A}^{T}(s,t_0)C^{T}(s)C(s)\Phi_{A}(s,t_0)\;ds.
\end{equation}
where $\Phi_{A}(s,t_0)$ is the transition matrix of \eqref{control1a}. 

A concept that is more restrictive than complete observability, and which in turn implies it, corresponds to the case where $t_0=t$ and $t_f=t+\sigma$, for a fixed $\sigma>0$, which corresponds to the \textit{uniform complete observability} (UCO), whose formal definition will be given later (see \cite[Definition 3.3.3]{Ioannou10.5555/211527} and \cite[Section 1.5]{sastry10.5555/63437} and \cite[Section 2]{SilvermanAnderson} for more detailed references).

\subsection{Detectability} The study of detectability for linear systems has a long tradition, closely tied to the design of observers and the stabilization of dynamical systems via output feedback. In the time-invariant setting, detectability is fully characterized by algebraic conditions and its equivalence to the existence of an exponentially convergent observer is well established. However, when moving to linear time-varying systems, the problem becomes significantly more intricate: the explicit time dependence of the dynamics demand a reformulation of both observability and detectability in terms of the state-transition operator
and associated Lyapunov and Gramian criteria.

By considering the system \eqref{control1a}-\eqref{control1b}, the exponential detectability consists of guaranteeing the existence of an observer $L:\mathbb{R}\to M_{n\times p}(\mathbb{R})$ described by the equation
\begin{equation*}
\dot{\hat{x}}(t)=A(t)\hat{x}(t)+L(t)[y(t)-C(t)\hat{x}(t)],
\end{equation*}
such that the linear system
\begin{equation}
        \label{SistError}
        \dot{e}(t)=[A(t)-L(t)C(t)]e(t)
    \end{equation}
will be exponentially stable, with $e(t)=x(t)-\dot{\hat{x}}(t)$ being the estimation error (see \cite{RAVI1992455}). Given the nature of different stability behaviors in the time-varying framework, uniform exponential detectability is defined based on the uniform exponential stability of the system \eqref{SistError} (see \cite{Tranninger}). For further details on the notion of detectability in this context, we refer the reader to 
\cite{andersonmoore, tranninger2}.

\subsection{Novelty of this work}
In this paper we introduce and characterize the concept of nonuniform complete observability (NUCO), which is more general than the classical Kalman property of uniform complete observability, but more restrictive than complete observability. We also present the corresponding nonuniform exponential detectability idea (NUED), related to the existence of an nonuniformly exponentially stable observer.

The first novel result corresponds to Theorems \ref{adjuntoNU} and \ref{dualNU}, in which, based on the concepts of adjoint and dual systems, a close relationship on nonuniform complete controllability and nonuniform complete observability is established, allowing to flow from one to the other and their respective consequences. Then, Theorem \ref{p2l3obsNOUNIFORME} follows directly, extending a well-known uniform property related to the observability.

The second novel result of this work is Theorem \ref{ObsImplicaDetect}, which states that NUCO implies NUED. Several intermediate and auxiliary results complete the work, including examples of NUCO systems and a counterxample of the reciprocal of Theorem \ref{ObsImplicaDetect}.

\subsection{Structure of the article}
Section \ref{sec-controlability} reviews the main results on uniform and nonuniform complete controllability, including the Gramian characterization and the relationship with the bounded growth property.  In Section \ref{sec-observability}, the idea of complete observability is introduced, in both uniform and nonuniform senses. After a brief review on dual and adjoint systems, the main relationships between nonuniform controllability and nonuniform observability are presented. Section \ref{sec-examples} presents some examples showing that nonuniform complete observability is effectively an intermediate stage between uniform complete observability and complete observability. Nonuniform exponential detectability is addressed in Section \ref{sec-detectability}, along with the idea of nonuniform exponential stability. Several auxiliary results are presented here and are used in Section \ref{sec-main-result}, where the main result of the article is developed, along with a counterexample of its reciprocal. Finally, some conclusions and possible research lines are commented.

\section{Controllability: basic notions and preliminaries}\label{sec-controlability}

\subsection{Uniform framework}
The study of the controllability of a control system has important consequences either at the level of stabilizing a system or establishing a relationship with observability and there by moving from one concept to another to address different approaches. In this sense, this section will raise the most relevant aspects of the concept of complete controllability, uniform complete controllability and a new property recently introduced in \cite{HMR}, namely, nonuniform complete controllability. The whole section will serve as an input to explore a new property of observability.

The concept of controllability was introduced by R. Kalman in \cite{Kalman} for the linear control
system
\begin{equation}
    \label{LTVControlabilidad}
    \dot{x}(t)=A(t)x(t)+B(t)u(t).
\end{equation}
The following definition corresponds to a reminder of the most important aspects of this concept (see \cite{Anderson,Sontag}):
\begin{definition}
The state $x_{0}\in \mathbb{R}^{n}$ of the control system \eqref{LTVControlabilidad} is controllable at time $t_{0}\in\mathbb{R}$, if there exists an input $u\colon [t_{0},t_{f}]\to \mathbb{R}^{p}$ such that $x(t_{f},t_{0},x_{0},u)=0$, where 
\begin{equation*}\label{eq:solution}
x(t,t_{0},x_0,u)=\Phi_{A}(t,t_{0})x_0+\int_{t_0}^t\Phi_{A}(t,\tau)B(\tau)u(\tau)d\tau   
\end{equation*}
is the solution of \eqref{LTVControlabilidad}. In addition, the control system \eqref{LTVControlabilidad} is: 
\begin{itemize}
\item[a)] \textbf{Controllable at time $t_{0}\in\mathbb{R}$} if any state $x_{0}$ is controllable at time $t_{0}\in\mathbb{R}$,
\item[b)] \textbf{Completely controllable} \textnormal{(CC)} if it is controllable at any time $t_{0}\in\mathbb{R}$.
\end{itemize}
\end{definition}

There is another perspective to establish the above concepts, which is given by the Gramian matrix of controllability, usually defined by
\begin{equation*}
W(a,b)=\displaystyle\int_{a}^{b}\Phi_{A}(a,s)B(s)B^{T}(s)\Phi_{A}^{T}(a,s)\;ds.
\end{equation*}

Specifically, the control system \eqref{LTVControlabilidad} is {\it controllable at time $t_{0}\in\mathbb{R}$} if and only if there exists
$t_{f}>t_{0}$ such that $W(t_{0},t_{f})>0$. Moreover, is {\it completely controllable} if and only if for 
any $t_{0}\in\mathbb{R}$, there exists $t_{f}>t_{0}$ such that $W(t_{0},t_{f})>0$. We refer the reader to \cite{Anderson67,Kreindler} for a detailed description of this topics.

A stronger property of controllability, also due to R. Kalman \cite{Kalman,Kalman69}, is given by the \textit{uniform
complete controllability} (UCC). In this case, there exists a parameter $\sigma>0$ such that “one can always
transfer $x$ to 0 and 0 to $x$ in a finite length $\sigma$ of time; moreover, such a transfer can never take place
using an arbitrarily small amount of control energy” \cite[p.157]{Kalman}. Kalman stated the property in
terms of the Gramian. 

\begin{definition} \label{controluniforme}   The linear control system \eqref{LTVControlabilidad} is said to be uniformly completely controllable on $J\subseteq \mathbb{R}$ if there exists a fixed constant $\sigma>0$ and positive numbers $\alpha_0(\sigma)$, $\beta_{0}(\sigma)$, $\alpha_1(\sigma)$ and $\beta_1(\sigma)$ such that the following relations hold for all $t\in J$: 
\begin{equation}
\label{UC1}
0<\alpha_0(\sigma)I\leq W(t,t+\sigma)\leq \alpha_1(\sigma)I, 
\end{equation}
\begin{equation}
\label{UC2}
0<\beta_0(\sigma)I\leq K(t,t+\sigma) \leq \beta_1(\sigma)I,
\end{equation}
where 
\begin{equation}\label{definicionK}
K(t,t+\sigma)=\Phi_{A}(t+\sigma,t)W(t,t+\sigma)\Phi_{A}^{T}(t+\sigma,t).
\end{equation}
\end{definition}

We emphasize the fact that the set $J$ can be thought of as $\mathbb{R}_{0}^{+}$, $\mathbb{R}_{0}^{-}$ or $\mathbb{R}$, being a distinguished element throughout this paper.
The property of uniform complete controllability have noticeable consequences, which have been pointed out by R. Kalman
in \cite{Kalman} and are summarized by the following result, whose proof is sketched in \cite[p.157]{Kalman} (see also \cite{Silverman}):
\begin{proposition}
\label{Prop1}
If the linear control system \eqref{LTVControlabilidad} is \textnormal{UCC}, then: 
\begin{itemize}
\item[i)] The inequalities \eqref{UC1}--\eqref{UC2} are also verified for any $\sigma'>\sigma$. 
\item[ii)] There exists a function $\alpha\colon [0,+\infty)\to (0,+\infty)$, with $\alpha(\cdot)\in\mathcal{B}$\footnote{$\mathcal{B}$ denote the set of functions $\alpha\colon \mathbb{R}\to \mathbb{R}$ mapping bounded sets into bounded sets.}, such that the plant of system \eqref{LTVControlabilidad} has a transition matrix verifying the property
\begin{equation}
\label{BG}
\|\Phi_{A}(t,s)\| \leq \alpha(|t-s|) \quad \textnormal{for all $t,s \in J$}.
\end{equation}
\end{itemize}
\end{proposition}

The above result deserves some comments:
\begin{remark}
The statement \textnormal{i)} of Proposition \textnormal{\ref{Prop1}} implies the existence of positive functions
$\alpha_{0},\alpha_{1},\beta_{0},\beta_{1}\colon [\sigma,+\infty) \to (0,+\infty)$ related to
\eqref{UC1} and \eqref{UC2}. These maps are used to construct the map $\alpha(\cdot)$ in \eqref{BG}.

\end{remark}

\begin{remark}
Kalman's article \textnormal{\cite{Kalman}} does not provide additional properties for the map $\alpha(\cdot)$ from \eqref{BG}. However, as done in the work of B. Zhou \textnormal{\cite[Lemma 4]{Zhou-21}} the condition \eqref{BG} with $\alpha(\cdot)\in \mathcal{B}$ is referred as the \textbf{Kalman's condition}. 
\end{remark} 

Based on Lemma 6 of \cite{HMR}, there is an equivalence between the Kalman's condition and the property of bounded growth, which is detailed as follows:

\begin{definition} 
 \label{UBGT}
The linear system \eqref{control1a} has a uniform bounded growth on the interval $J\subseteq\mathbb{R}$ if there exist constants $K_{0}>0$ and $a>0$ such that its transition matrix satisfies
\begin{equation}
\label{PhiboundUniform}
\|\Phi_{A}(t,\tau)\|\leq K_{0}e^{a|t-\tau|}\quad \textnormal{for any $t,\tau\in J$}.
\end{equation}
\end{definition}

The next result, whose proof has been sketched by Kalman in \cite[p.157]{Kalman}, des\-cribes a more surprising relation between the properties of uniform complete controllability and the Kalman's condition:
\begin{proposition}
\label{Prop2}
If any two of the properties \eqref{UC1}, \eqref{UC2} and \eqref{BG} hold, the remaining one is also true.
\end{proposition}

\subsection{Notion of nouniform complete controllability on $\mathbb{R}_{0}^{+}$}
By considering the proposition \ref{Prop2}, it is sufficient to consider an example of a linear system where the Kalman condition is not verified in order to obtain a system that does not admit uniform complete controllability. From this point of view, on \cite[p. 8]{HMR} it is proved that the system 
$$\dot{x}=[\lambda_0+at\sin(t)]x,\quad \textnormal{with}\; \lambda_0<a<0,$$
does not satisfy the Kalman condition, while this system verify another property, namely, \textit{nonuniform bounded growth}, which corresponds to a generalization of the Kalman condition and, in a way, is the starting point for defining a new type of controllability that adapts to systems that do not necessarily satisfy the Kalman condition.

\begin{definition} 
 \label{NUBGT}
The linear system \eqref{control1a} has a nonuniform bounded growth on the interval $J\subseteq \mathbb{R}$ if there exist constants $K_{0}>0$, $a>0$ and $\eta>0$ such that its transition matrix satisfies
\begin{equation}
\label{Phibound}
\|\Phi_{A}(t,\tau)\|\leq K_{0}e^{\eta |\tau| }e^{a|t-\tau|}\quad \textnormal{for any $t,\tau\in J$}.
\end{equation}
\end{definition}




%


In view of this, in \cite{HMR} a new type of controllability, called \textit{nonuniform complete controllability} (NUCC), is established. In this definition was consider $J=\mathbb{R}_{0}^{+}$.

\begin{definition}
\label{DEFNUCC}
The linear control system \eqref{LTVControlabilidad} is said to be nonuniformly completely controllable on $J=\mathbb{R}_{0}^{+}$ if there exist fixed numbers $\mu_0\geq0$, $\mu_1\geq0$, $\tilde{\mu}_0\geq0$, $\tilde{\mu}_1\geq0$ and functions $\alpha_0(\cdot), \beta_{0}(\cdot),\alpha_1(\cdot),\beta_1(\cdot):[0,+\infty)\to(0,+\infty)$ such that for any $t\in J$, there exists $\sigma_0(t)>0$ with:
\begin{equation}
\label{alpha0alpha1}
    0<e^{-2\mu_0 t}\alpha_0(\sigma)I\leq W(t,t+\sigma)\leq e^{2\mu_1 t}\alpha_1(\sigma)I,
\end{equation}
\begin{equation}
\label{K}
    0<e^{-2\tilde{\mu}_0 t}\beta_0(\sigma) I\leq K(t,t+\sigma)\leq e^{2\tilde{\mu}_1 t}\beta_1(\sigma) I, 
\end{equation}
for every $\sigma\geq \sigma_{0}(t)$. 
\end{definition}

Taking into account the relationship between uniform complete controllability and the Kalman condition \cite{Kalman}, then let us consider the following definition, which will be coupled to the definition of nonuniform complete controllability.

\begin{lemma}
If the control system \eqref{LTVControlabilidad} is nonuniformly completely controllable, then there exist $\nu>0$ and a function  $\alpha(\cdot)\in\mathcal{B}$ satisfying
\begin{equation}
\label{crec-acot}
\|\Phi_{A}(t,\tau)\|\leq e^{\nu|\tau|}\,\alpha(|t-\tau|)  \quad \textnormal{for any $t,\tau\in\mathbb{R}$}.
\end{equation}
\end{lemma}

\begin{remark}
\label{Kalmanproperty}
    As introduced in \textnormal{\cite{HMR}}, the expression \eqref{crec-acot} is known as the nonuniform Kalman property. Additionally, we can note that the nonuniform bounded growth is a particular case of the \textbf{nonuniform Kalman property}.
\end{remark}

From the nonuniform complete controllability and the nonuniform Kalman condition, we finish this section by considering the following result about the relationship between these concepts (see \cite[Theorem 25]{HMR} for a detailed proof by considering $J=\mathbb{R}_{0}^{+}$).

\begin{theorem}
\label{p2l3}
Any two of the properties \eqref{alpha0alpha1}, \eqref{K} and \eqref{crec-acot} (on $J=\mathbb{R}_{0}^{+}$)
imply the third one.
\end{theorem}

The work in \cite{HMR} was raised in a context where only the relationship of nonuniform complete controllability and nonuniform complete stabilizability on $\mathbb{R}_{0}^{+}$ is studied. However, a careful reading of section 3.1 in \cite{HMR} allows us to observe that all the results can be generalized to the $J=\mathbb{R}$ case. First, nonuniform complete controllability in R is defined.

\begin{definition}
\label{DEFNUCCenR}
The linear control system \eqref{LTVControlabilidad} is said to be nonuniformly completely controllable on $J=\mathbb{R}$ if we have the same constant than Definition \ref{DEFNUCC} such that for any $t\in J$, there exists $\sigma_0(t)>0$ with:
\begin{equation}
\label{alpha0alpha1enR}
    0<e^{-2\mu_0 |t|}\alpha_0(\sigma)I\leq W(t,t+\sigma)\leq e^{2\mu_1 |t|}\alpha_1(\sigma)I,
\end{equation}
\begin{equation}
\label{KenR}
    0<e^{-2\tilde{\mu}_0 |t|}\beta_0(\sigma) I\leq K(t,t+\sigma)\leq e^{2\tilde{\mu}_1 |t|}\beta_1(\sigma) I, 
\end{equation}
for every $\sigma\geq \sigma_{0}(t)$. 
\end{definition}

Based on the previous definition, we can state the following theorem which is a generalization of Theorem \ref{p2l3}.

\begin{theorem}
\label{p2l3ControlR}
Any two of the properties \eqref{alpha0alpha1enR}, \eqref{KenR} and \eqref{crec-acot} (on $J=\mathbb{R}$)
imply the third one.
\end{theorem}

\section{Observability framework}\label{sec-observability}

\subsection{Uniform complete observability and its relation with the uniform complete controllability} 
Taking into account the observability Gramian described in \eqref{gramianobservability}, we begin this section by recalling the definition of uniform complete observability (see \cite{Anderson69}, \cite[Definition 2.6]{Ikeda2} and \cite{Kalman}). 

\begin{definition}\label{UCO}    The linear time-varying control system \eqref{control1a}-\eqref{control1b} is said to be uniformly completely observable \textnormal{(UCO)} on $J=\mathbb{R}_{0}^{+}$ if there exists a fixed constant $\sigma>0$ and positive numbers $\bar{\alpha}_0(\sigma)$, $\bar{\beta}_0(\sigma)$, $\bar{\alpha}_1(\sigma)$ and $\bar{\beta}_1(\sigma)$ such that the following relations hold for all $t\geq 0$: 
\begin{equation}
\label{UCO1}
0<\bar{\alpha}_0(\sigma)I\leq M(t,t+\sigma)\leq \bar{\alpha}_1(\sigma)I,
\end{equation}
\begin{equation}
\label{UCO2}
0<\bar{\beta}_0(\sigma)I\leq N(t,t+\sigma)\leq \bar{\beta}_1(\sigma)I,
\end{equation}
where 
\begin{equation}
\label{UC02N}
   N(t,t+\sigma)=\Phi_{A}^{T}(t,t+\sigma)M(t,t+\sigma)\Phi_{A}(t,t+\sigma) 
\end{equation}
\end{definition}

It is usual to find in the literature (see \cite{BATISTA20173598}) that if $A(\cdot)$ and $C(\cdot)$ are initially considered to be bounded matrices, then in the previous definition we only requires the following condition: 
$$\bar{\alpha}_0(\sigma)I\leq M(t,t+\sigma),$$
since the second inequality of \eqref{UCO1} and the inequalities in \eqref{UCO2} are obtained as a consequence of the bounding of the matrices $A(\cdot)$ and $C(\cdot)$.

In order to pave the way for detailing the close relationship between observability and controllability, we consider the \textit{adjoint system} and the \textit{dual system} (see \cite{Callier,IlchmannNote} for more details) of 
\eqref{control1a}-\eqref{control1b} described respectively by \eqref{traspuesto1a} and \eqref{dual1a}:
\begin{equation}
\label{traspuesto1a}
\dot{x}(t)=-A^{T}(t)x(t)-C^{T}(t)u(t), 
\end{equation}
\begin{equation}
    \label{dual1a}
\dot{x}(t)=A^{T}(-t)x(t)+C^{T}(-t)u(t),
\end{equation}
with $t\in J\subseteq \mathbb{R}$, $x(t)\in\mathbb{R}^{n}$, $u(t)\in\mathbb{R}^{m}$. The dimensions of $A(t)$ and $C(t)$ are $n\times n$ and $m\times n$ respectively. In addition, the transition matrix of the plant of \eqref{traspuesto1a} and \eqref{dual1a} are denoted by $\Psi_{a}(\cdot,\cdot)$ and $\Psi_{d}(\cdot,\cdot)$ respectively, which satisfy the following relation with the transition matrix of the system \eqref{control1a}:
\begin{equation}
\label{Psiadjunto}
    \Psi_{a}(t,\tau)=\Phi_{A}^{T}(\tau,t),
\end{equation}
\begin{equation}
\label{Psidual}
    \Psi_{d}(t,\tau)=\Phi_{A}^{T}(-\tau,-t).
\end{equation}

It is important to emphasize that depending on the domain $J$ imposed for the control system \eqref{control1a}-\eqref{control1b}, the domain for the adjoint and dual system will be determined. For example, in the case where $J=\mathbb{R}_0^{+}$, we will have that the domains for the adjoint and dual systems will be $J_a=\mathbb{R}_{0}^{+}$ and $J_d=\mathbb{R}_{0}^{-}$ respectively. At some point in this paper we will consider the case where $J=J_a=J_d=\mathbb{R}$.


The following two theorems establish the relationship between the Gramian matrices of the system \eqref{control1a}-\eqref{control1b} and the respective Gramian matrices of the systems \eqref{traspuesto1a} and \eqref{dual1a}.

\begin{theorem} 
\label{adjuntoydual}
By considering $W_{a}$, $K_a$, $W_{d}$ and $K_d$ as the Gramian matrices of controllability of the systems \eqref{traspuesto1a} and \eqref{dual1a} respectively, then the following identities related to the Gramian matrices of observability $M$ and $N$ of the system \eqref{control1a}-\eqref{control1b} are satisfied:
\begin{equation} 
\label{identidadM}
M(t,t+\sigma)=W_{a}(t,t+\sigma)=K_{d}(-t-\sigma,-t).
\end{equation}
\begin{equation}
\label{identidadN}
 N(t,t+\sigma)=K_{a}(t,t+\sigma)=W_{d}(-t-\sigma,-t).
\end{equation}

\end{theorem}
\begin{proof}
    The Gramian matrix of observability of the system (\ref{control1a})-(\ref{control1b}) is given by \eqref{gramianobservability}, with $t_0=t$ and $t_f=t+\sigma$. In addition, by considering the adjoint system and \eqref{Psiadjunto}, then we have that:
    \begin{equation}
    \label{RelacionObsCont}
    \begin{array}{rcl}
    &&M(t,t+\sigma)=\displaystyle\int_{t}^{t+\sigma} \Phi_{A}^{T}(s,t)C^{T}(s)C(s)\Phi_{A}(s,t)\; ds,\\
    &=&\displaystyle\int_{t}^{t+\sigma} \Psi_{a}(t,s)(-C^{T}(s))(-C(s))\Psi_{a}^{T}(t,s)\; ds,\\
    &=&W_{a}(t,t+\sigma),
    \end{array}
    \end{equation}
and
    \begin{equation}
    \label{RelacionObsCont2}
    \begin{array}{rcl}
&&N(t,t+\sigma)=\Phi_{A}^{T}(t,t+\sigma)M(t,t+\sigma)\Phi_{A}(t,t+\sigma),\\
&=&\displaystyle\int_{t}^{t+\sigma} \Phi_{A}^{T}(s,t+\sigma)C^{T}(s)C(s)\Phi_{A}(s,t+\sigma)\; ds,\\
    &=&\displaystyle\int_{t}^{t+\sigma} \Psi_{a}(t+\sigma,s)(-C^{T}(s))(-C(s))\Psi_{a}^{T}(t+\sigma,s)\; ds,\\\\
    &=&\Psi_{a}(t+\sigma,t)W_{a}(t,t+\sigma)\Psi_{a}^{T}(t+\sigma,t)=K_{a}(t,t+\sigma).
    \end{array}
    \end{equation}
On the other hand, by considering the dual system and \eqref{Psidual}, we obtain that: 

\begin{equation}
    \label{RelacionObsContDual1}
    \begin{array}{rcl}
    &&M(t,t+\sigma)=\displaystyle\int_{t}^{t+\sigma} \Phi_{A}^{T}(s,t)C^{T}(s)C(s)\Phi_{A}(s,t)\; ds,\\
    &=&\displaystyle\int_{t}^{t+\sigma} \Psi_{d}(-t,-s)C^{T}(s)C(s)\Psi_{d}^{T}(-t,-s)\; ds,\\
    &=&\displaystyle-\int_{-t}^{-t-\sigma} \Psi_{d}(-t,r)C^{T}(-r)C(-r)\Psi_{d}^{T}(-t,r)\; dr,\\
    &=&\displaystyle\int_{-t-\sigma}^{-t} \Psi_{d}(-t,r)C^{T}(-r)C(-r)\Psi_{d}^{T}(-t,r)\; dr,\\
    &=&\Psi_{d}(-t,-t-\sigma) W_{d}(-t-\sigma,-t)\Psi_{d}^{T}(-t,-t-\sigma),\\
    &=&K_{d}(-t-\sigma,-t)
    \end{array}
    \end{equation}
    and 
    \begin{equation}
    \label{RelacionObsDual2}
    \begin{array}{rcl}
&&N(t,t+\sigma)=\Phi_{A}^{T}(t,t+\sigma)M(t,t+\sigma)\Phi_{A}(t,t+\sigma),\\
&=&\displaystyle\int_{t}^{t+\sigma} \Phi_{A}^{T}(s,t+\sigma)C^{T}(s)C(s)\Phi_{A}(s,t+\sigma)\; ds,\\
    &=&\displaystyle\int_{t}^{t+\sigma} \Psi_{d}(-t-\sigma,-s)C^{T}(s)C(s)\Psi_{d}^{T}(-t-\sigma,-s)\; ds,\\
    &=&\displaystyle\int_{-t-\sigma}^{-t} \Psi_{d}(-t-\sigma,r)C^{T}(-r)C(-r)\Psi_{d}^{T}(-t-\sigma,r)\; dr,\\
    &=&W_{d}(-t-\sigma,-t).
    \end{array}
    \end{equation}
\end{proof}


Based on the Theorem \ref{adjuntoydual}, we have the ability to identify controllability and observability results depending on the system under consideration. With this same approach, it is straightforward to obtain a result such as the one stated in Proposition \ref{Prop2}, where the uniform Kalman property is also related to uniform complete observability (see \cite{Anderson69,Kalman} for details). 

\begin{theorem}
\label{p2l3obs}
    In the uniform context, any two of the properties \eqref{UCO1}, \eqref{UCO2} and \eqref{BG} imply the third one.
\end{theorem}

\subsection{Nonuniform complete observability}
As with the nonuniform complete controllability property, if there is an intermediate property between complete observability and uniform complete observability, it must be closely related to the nonuniform Kalman property. The definition of this new property is given below. 

\begin{definition}
The system \eqref{control1a}-\eqref{control1b} is nonuniformly completelly observable \textnormal{(NUCO)} on $J\subseteq \mathbb{R}$ if there exist fixed numbers $\nu_0\geq0$, $\nu_1\geq0$, $\bar{\nu}_0\geq0$, $\bar{\nu}_1\geq0$ and functions $\vartheta_0$, $\varrho_0$, $\vartheta_1$ and $\varrho_1:[0,+\infty)\to (0,+\infty)$ such that for any $t\in J$, there exists $\sigma_0(t)>0$ with:
\begin{equation}
\label{Mobserv}
e^{-2\nu_0 |t|}\vartheta_0(\sigma)I\leq M(t,t+\sigma)\leq  e^{2\nu_1 |t|}\vartheta_1(\sigma)I
\end{equation}
\begin{equation}
\label{Mobserv2}
e^{-2\bar{\nu}_0 |t|}\varrho_0(\sigma)I\leq N(t,t+\sigma)\leq  e^{2\bar{\nu}_1 |t|}\varrho_1(\sigma)I
\end{equation}
for every $\sigma\geq\sigma_0(t)$. 
\end{definition}

\begin{remark}
     \label{Robserv}
The similarities between observability and controllability in this nonuniform framework are evident (see Definition \textnormal{\ref{DEFNUCCenR}}), which allows us to note that this notion is more general than uniform complete observability and more restrictive than complete observability.

\end{remark}

In Theorem \ref{adjuntoydual}, the relationship between the Gramian matrices of the controllability and observability were shown, with emphasis on whether the analysis is performed on the adjoint system or on the dual system. Thus, we establish the following theorems.  

\begin{theorem} 
\label{adjuntoNU}
    The system \eqref{control1a}-\eqref{control1b} is \textnormal{NUCO} if and only if the adjoint system \eqref{traspuesto1a} is \textnormal{NUCC}. 
\end{theorem}

\begin{proof}
Based on the identities \eqref{identidadM}-\eqref{identidadN} associated to the adjoint system, we have that $M(t,t+\sigma)=W_{a}(t,t+\sigma)$ and $N(t,t+\sigma)=K_{a}(t,t+\sigma)$, which allow us to ensure that the inequalities \eqref{alpha0alpha1enR}-\eqref{KenR} are equivalent to \eqref{Mobserv}-\eqref{Mobserv2} and the theorem follows.      
\end{proof}

\begin{theorem} 
\label{dualNU}
    The system \eqref{control1a}-\eqref{control1b} is \textnormal{NUCO} if and only if the dual system \eqref{dual1a} is \textnormal{NUCC}. 
\end{theorem}
\begin{proof}
    Based on the identities \eqref{identidadM}-\eqref{identidadN} associated to the dual system, we have that $M(t,t+\sigma)=K_{d}(-t-\sigma,-t)$ and $N(t,t+\sigma)=W_{d}(-t-\sigma,-t)$. Except for exchanges in functions and constants ($e^{-2\mu_0|t|}\alpha_0(\sigma)$ by $e^{-2\nu_0|t|}\varrho_0(\sigma)$ and $e^{2\mu_1|t|}\alpha_1(\sigma)$ by $e^{2\nu_1|t|}\varrho_1(\sigma)$), the inequalities \eqref{alpha0alpha1enR}-\eqref{KenR} are equivalent to \eqref{Mobserv}-\eqref{Mobserv2} and the theorem follows. 
\end{proof}
\begin{figure*}[h!]
\centering
\begin{tikzpicture}
[squarednode/.style={rectangle, draw=black!60, fill=red!5, very thick, minimum size=25mm},]

  \draw[line width=5pt]

    (0,0)  node[squarednode] (a) {$\begin{array}{l}
     \textnormal{NUCO}   \\
     \\
\left\{\begin{array}{rcl}
\dot{x}(t)&=&A(t)x(t)\\
y(t)&=&C(t)x(t)
\end{array}\right.
\end{array}$}

    (8,0)  node[squarednode]                 (b) {$\begin{array}{l}
     \textnormal{NUCC}  \\
     \\
\dot{x}(t)=-A^{T}(t)x(t)-C^{T}(t)u(t)
\end{array}$}

    (0,-4) node[squarednode]         (c) {$\begin{array}{l}
     \textnormal{NUCO}  \\
     \\
\left\{\begin{array}{rcl}
\dot{x}(t)&=&-A(-t)x(t)\\
y(t)&=&-C(-t)x(t)
\end{array}\right.
\end{array}$}

    (8,-4)  node[squarednode]                 (d) {$\begin{array}{l}
     \textnormal{NUCC}  \\
     \\
\dot{x}(t)=A^{T}(-t)x(t)+C^{T}(-t)u(t)
\end{array}$};


%


  \draw[<->,red,ultra thick] (a) -- (b)
  node[midway,above] {Th.\ref{adjuntoNU}};

  \draw[<->,red,ultra thick] (a) -- (d)
    node[midway,above] {$\quad$ Th.\ref{dualNU}};
 
  \draw[<->,red,ultra thick] (c) -- (d)
    node[midway,above] {Th.\ref{adjuntoNU}};

  \draw[<->] (a) -- (c);
  \draw[<->] (b) -- (d);

\end{tikzpicture}

\caption{NUCO and NUCC relations between a given system (above-left), its adjoint (above-right) and its dual (below-right). It is also shown the adjoint of the dual (below-left). Red thick arrows represent the main results of Section 3. Black thin arrows shown corollaries of these results, not stated in the text, since they are not used in this work.}
\label{fig-NUCO-NUCC}
\end{figure*}
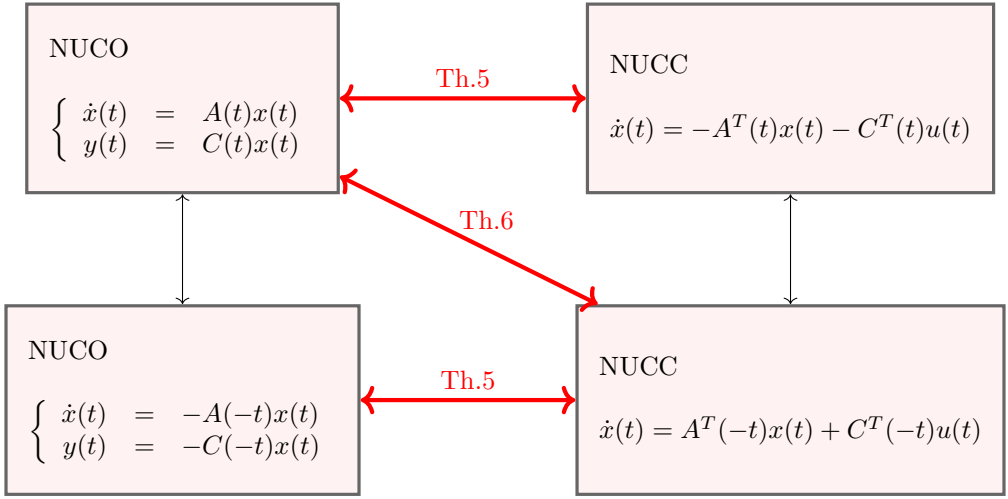

The Figure \ref{fig-NUCO-NUCC} graphically shows the results of Theorems \ref{adjuntoNU} and \ref{dualNU}, and illustrates how they will be used in this work.  

\begin{remark}
\label{relacionintervalos}
    Looking carefully at the above theorems, we can notice that in Theorem \ref{adjuntoNU} if the system \eqref{control1a}-\eqref{control1b} is \textnormal{NUCO} on $J=\mathbb{R}_{0}^{+}$, that will be equivalent to the adjoint system \eqref{traspuesto1a} being \textnormal{NUCC} on the same set $J=\mathbb{R}_{0}^{+}$. On the other hand, for Theorem \ref{dualNU}, if the system \eqref{control1a}-\eqref{control1b} is \textnormal{NUCO} in $J=\mathbb{R}_{0}^{+}$, that will be equivalent to the dual system \eqref{dual1a} being \textnormal{NUCC} in $\tilde{J}=\mathbb{R}_{0}^{-}$.
\end{remark}

\begin{remark}
\label{RelacionContObsUniforme}
    As a direct consequence of Theorems \ref{adjuntoNU} and \ref{dualNU}, we can first establish that the system \eqref{control1a}-\eqref{control1b} is \textnormal{UCO} if and only if the adjoint system \eqref{traspuesto1a} is \textnormal{UCC} and, on the other hand, the system \eqref{control1a}-\eqref{control1b} is \textnormal{UCO} if and only if the dual system \eqref{dual1a} is \textnormal{UCC}.
\end{remark}

The close relationship between controllability and observability in a nonuniform sense, established from the previous two theorems, in addition to including the nonuniform Kalman condition, allows us to establish the following theorem related to observability, which is very similar to the Theorem \ref{p2l3ControlR}.

\begin{theorem}
\label{p2l3obsNOUNIFORME}
    Any two of the properties \eqref{crec-acot}, \eqref{Mobserv} and \eqref{Mobserv2} (on $J=\mathbb{R}$) imply the third one.
\end{theorem}

\section{Examples}\label{sec-examples}
In this section different examples will be explored, where more details about the new property of nonuniform complete observability and its relation with complete observability and uniform complete observability will be given. 


\subsection{Example 1} Based on previously proved relationship between controllability and observability, this example is strongly focused on the nonuniform bounded growth on $J=\mathbb{R}$ property and follows the same line as the example presented in \cite{HMR} (see Subsection 3.2).
If we assume that the linear system (\ref{control1a}) admits the property of nonuniform bounded growth with parameters $K_0$, $a$ and $\eta$, it is interesting to explore some conditions on $C(t)$ leading to the estimates (\ref{Mobserv}) and \eqref{Mobserv2}.
In order to get them, we will assume that there exist $\gamma_0$, $\gamma_1>0$ and $c_0$, $c_1>0$ such that for any $t\in\mathbb{R}$: 
    \begin{equation}
    \label{Cexample}
    c_0e^{-2\gamma_0 |t|}I\leq C^{T}(t)C(t)\leq c_1e^{2\gamma_1 |t|}I.
    \end{equation}
    Based on the transition matrix property, we have that for any vector $x\neq0$, it is verified that  
$
x=\Phi_{A}^{T}(s,t)\Phi_{A}^{T}(t,s)x
$
which implies the following:    
    $$|x|^{2}\leq\left\|\Phi_{A}^{T}(t,s)\right\|^{2}|\Phi_{A}^{T}(s,t)x|^{2}\Rightarrow \frac{|x|^{2}}{\left\|\Phi_{A}(t,s)\right\|^{2}}\leq|\Phi_{A}^{T}(s,t)x|^{2}, $$
    and by considering the nonuniform bounded growth hypothesis, we can ensure that
    \begin{equation*}
    \left\|\Phi_{A}(t,s)\right\|\leq K_0e^{a|s-t|+\eta |s|}\Leftrightarrow\frac{1}{K_0e^{a|s-t|+\eta |s|}}\leq \frac{1}{\left\|\Phi_{A}(t,s)\right\|}.
    \end{equation*}

    From the above, we will obtain the right-hand estimate of (\ref{alpha0alpha1}). For any $t\in\mathbb{R}$, there exist $\sigma_0(t)>0$ such that for any $\sigma\geq\sigma_0(t)$ and $x\neq0$:

\begin{displaymath}
    \begin{array}{rcl}
    &&\displaystyle\int_{t}^{t+\sigma}x^{T}\Phi_{A}^{T}(s,t)C^{T}(s)C(s)\Phi_{A}(s,t)x\;ds\\
    &\leq& \displaystyle\int_{t}^{t+\sigma}K_0^{2}e^{2a|s-t|+2\eta |t|}c_1^{2}e^{2\gamma_1 |s|}|x|^{2}\;ds,\\
    &\leq& \displaystyle\int_{t}^{t+\sigma}K_0^{2}e^{2(a+\gamma_1)|s-t|}e^{2\eta |t|}c_1^{2}e^{2\gamma_1 |t|}|x|^{2}\;ds,\\
    &=&\displaystyle K_0^{2}c_1^{2}|x|^{2}e^{-2(a+\gamma_1)t}e^{2(\eta+\gamma_1)|t|}\int_{t}^{t+\sigma}e^{2(a+\gamma_1)s}\;ds,\\
     &=&\displaystyle e^{2(\eta+\gamma_1) |t|}\frac{K_0^{2}c_1^{2}|x|^2}{2(a+\gamma_1)}\left [e^{2(a+\gamma_1)\sigma}-1\right]
    \end{array}
    \end{displaymath}
which implies that
 \begin{displaymath}
x^{T}M(t,t+\sigma)x\leq 
    \displaystyle e^{2(\eta+\gamma_1) |t|}\frac{K_0^{2}c_1^{2}\left [e^{2(a+\gamma_1)\sigma}-1\right]}{2(a+\gamma_1)}|x|^2.
    \end{displaymath}

Similarly, we note that
    $$\begin{array}{rcl}
         &&\displaystyle\int_{t}^{t+\sigma}x^{T}\Phi_{A}^{T}(s,t)C^{T}(s)C(s)\Phi_{A}(s,t)x\;ds \\
         &\geq&  \displaystyle\int_{t}^{t+\sigma}\frac{1}{K_0^{2}}e^{-2a|s-t|-2\eta |s|}c_0^{2}e^{-2\gamma_0 |s|}|x|^{2}\;ds,\\
         &\geq&\displaystyle\int_{t}^{t+\sigma}\frac{1}{K_0^{2}}e^{-2(a+\eta+\gamma_0)|s-t|-2(\eta+\gamma_0) |t|}c_0^{2}|x|^{2}\;ds,\\
         &=&\displaystyle\frac{c_0^{2}|x|^2}{K_0^{2}}e^{2(a+\eta+\gamma_0)t}e^{-2(\eta+\gamma_0)|t|}\int_{t}^{t+\sigma}e^{-2(a+\eta+\gamma_0)s}ds,\\
         &=&\displaystyle e^{-2(\eta+\gamma_0)|t|} \frac{c_0^{2}|x|^2}{K_0^22(a+\eta+\gamma_0)}\left[1-e^{-2(a+\eta+\gamma_0)\sigma}\right]
    \end{array}$$
which implies
\begin{displaymath}
x^{T}M(t,t+\sigma)x \geq  \displaystyle e^{-2(\eta+\gamma_0)|t|} \frac{c_0^{2}\left[1-e^{-2(a+\eta+\gamma_0)\sigma}\right]}{K_0^22(a+\eta+\gamma_0)}|x|^2
\end{displaymath}
and the estimate \eqref{Mobserv} is satisfied by considering 
$$\vartheta_0(\sigma)=\displaystyle\frac{c_0^{2}}{K_0^22(a+\eta+\gamma_0)}\left[1-e^{-2(a+\eta+\gamma_0)\sigma}\right],$$
$$\vartheta_{1}(\sigma)=\displaystyle \frac{K_0^{2}c_1^{2}}{2(a+\gamma_1)}\left [e^{2(a+\gamma_1)\sigma}-1\right].$$

Finally, the Theorem \ref{p2l3obsNOUNIFORME} allow us to obtain the estimate \eqref{Mobserv2}, therefore the property of nonuniform complete observability is guaranteed.

As an extra comment, in the previous proof we can notice that $\sigma_0(t)=\sigma_0>0$ is constant and works in order to prove the inequalities for all $t\in\mathbb{R}$.

\subsection{Example 2} 
In \cite[p.5]{ASP} it was showed the scalar control system
\begin{subequations}
  \begin{empheq}[left=\empheqlbrace]{align}
    & \dot{x}=-\displaystyle\frac{x}{t}, \label{ex2a} 
    \\
    & y=x, \label{ex2b}
  \end{empheq}
\end{subequations}
with $t\geq1$, which verify the properties UCO, NUCO and CO on $[1,+\infty)$. In fact, it is easy to check that the transition matrix and Gramian of observability of the system \eqref{ex2a}-\eqref{ex2b} are $\Phi(t,\tau)=\frac{\tau}{t}$ and $M(t_0,t_f)=\frac{t_0 (t_f-t_0)}{t_f}$ respectively. In addition, by considering $\sigma=1$, we have that for any $t\geq1$
$$\frac{1}{2}<M(t,t+1)=\frac{t}{t+1}<1,$$
$$1<N(t,t+1)=\frac{t+1}{t}<2,$$
then UCO property is satisfied and consequently we have the property CC. In relation to the property NUCO, for any $t\geq1$, there exists $\sigma_0=1$ such that for all $\sigma\geq1$:
$$\frac{1}{2}e^{-2\sigma t}<\frac{1}{2}<M(t,t+1)=\frac{t}{t+1}<1<e^{2\sigma t},$$
$$e^{-2\sigma t}<1<N(t,t+1)=\frac{t+1}{t}<2<2e^{2\sigma t}.$$

Moreover, by considering the Figure \ref{fig-NUCO-NUCC}, we can also establish that the following system is NUCO on $(-\infty,-1]$:
\begin{subequations}
  \begin{empheq}[left=\empheqlbrace]{align}
    & \dot{x}=-\displaystyle\frac{x}{t}, \label{ex2aotro} 
    \\
    & y=-x. \label{ex2botro}
  \end{empheq}
\end{subequations}

\subsection{Example 3} In \cite[Subsection 3.3]{HMR} it was proved that the scalar control system
\begin{equation}
\label{ex-BV}
\dot{x}=-tx +\sqrt{2(t-1)}e^{-t+\frac{1}{2}}u(t)\quad \textnormal{with $t\geq 1$}
\end{equation}
is CC but is neither UCC nor NUCC on $[1,+\infty)$. On the other hand, the adjoint system of \eqref{ex-BV} is given by
\begin{subequations}
  \begin{empheq}[left=\empheqlbrace]{align}
    & \dot{x}=tx, \label{ex3a} \\
    & y=-\sqrt{2(t-1)}e^{-t+\frac{1}{2}}x, \label{ex3b}
  \end{empheq}
\end{subequations}

By considering Theorem \ref{adjuntoNU}, Theorem \ref{dualNU} and Remark \ref{RelacionContObsUniforme}, we can conclude that system \eqref{ex3a}-\eqref{ex3b} is CO but is neither UCO nor NUCO on $[1,+\infty)$. Additionally, it is possible ensure that the dual system 
\begin{subequations}
  \begin{empheq}[left=\empheqlbrace]{align}
    & \dot{x}=-tx, \label{ex3aotro} \\
    & y=-\sqrt{2(-t-1)}e^{t+\frac{1}{2}}x, \label{ex3botro}
  \end{empheq}
\end{subequations}
is CO but is neither UCO nor NUCO on $(-\infty,-1]$.
\subsection{Example 4} By considering the scalar control system
\begin{subequations}
  \begin{empheq}[left=\empheqlbrace]{align}
    & \dot{x}=-x\sin(x), \label{ex4a} \\
    & y=x, \label{ex4b}
  \end{empheq}
\end{subequations}
it is straightforward to deduce that its transition matrix is given by $$\displaystyle\Phi(t,s)=e^{t\cos(t)-s\cos(s)+\sin(s)-\sin(t)}$$
and additionally, we can deduce that
$$\|\Phi(t,s)\|\leq e^{2}e^{|t-s|}e^{2|s|},$$
this means that the system \eqref{ex4a} has the nonuniform bounded growth property. In addition, in this system we have that $C(t)=1$, then the inequality  \eqref{Cexample} is verified. Therefore, the conditions of the Example 1 are satisfied and this implies the NUCO property, which allow us to ensure the complete observability.

Now, we will verify that the system \eqref{ex4a}--\eqref{ex4b} cannot be UCO. Indeed, otherwise, 
the Gramian inequalities \eqref{UCO1}--\eqref{UCO2} will be verified
and Theorem \ref{p2l3obs} would implies the existence of $\alpha(\cdot)\in\mathcal{B}$ such that
$$
\Phi(t,s)=e^{t\cos(t)-s\cos(s)+\sin(s)-\sin(t)}\leq \alpha(|t-s|),
$$
which is equivalent to
$$
t\cos(t)-s\cos(s)+\sin(s)-\sin(t)=\ln\left(\alpha(|t-s|\right)).
$$

Now, let us consider $t=2n\pi$  and $s=\frac{3\pi}{2}+2(n-1)\pi$. Then, the above inequality is equivalent to:
$$
2n\pi-1\leq 2\ln\left(\alpha\left(\frac{\pi}{2}\right)\right).
$$

Now, by letting $n\to +\infty$, the above inequality leads to a contradiction.

\begin{remark}
This section allows us to note that indeed the property of \textnormal{NUCO} corresponds to a concept that is halfway between \textnormal{CO} and \textnormal{UCO}.
\end{remark}

\section{Detectability in a time-varying context}\label{sec-detectability}

In order to begin, we will establish the formal definition of the concept of uniform exponential stability.
\begin{definition}
    The linear system
    \begin{equation}
 \label{ALS}
\dot{x}=V(t)x
\end{equation}
is said to be uniformly exponentially stable \textnormal{(UES)} if there exist constants $\beta>0$ and $M\geq1$ such that for all $s\in\mathbb{R}$ and for all $t\geq s$:
$$\|\Phi_{V}(t,s)\|\leq Me^{-\beta(t-s)}.$$
\end{definition}


Formally, we have the following definition of uniform exponential detectability (see \cite{RAVI1992455,Tranninger}):
\begin{definition}
\label{DetExpNU}
    The system \eqref{control1a}-\eqref{control1b} is called uniformly exponentially detectable \textnormal{(UED)} if there exists a uniformly bounded output injection gain $L(t)$, also called observer, such that the linear system \eqref{SistError} is uniformly exponentially stable. 
\end{definition}

\begin{remark}
    It is interesting to note that due to the multiple definitions of stability existing in the nonautonomous framework, the concept of detectability can be adapted to the different properties of stability. Therefore, it motivates the idea that there is a definition of detectability that is suitable to the nonuniform framework, which will be studied in the following subsection. 
\end{remark}

Based on the above, we will now detail the definition of the concept of \textit{nonuniform exponential stability} and then relate this kind of stability with the respective stability properties of the adjoint and dual systems.

\begin{definition}
    The linear system \eqref{ALS} is said to be nonuniformly exponentially stable \textnormal{(NUES)} forward if there exist constant $\beta>0$, $\delta\geq0$  and $M\geq1$ such that for all $s\in\mathbb{R}$ and for all $t\geq s$:
\begin{equation}
\label{NUESInfinito}
\|\Phi_{V}(t,s)\|\leq Me^{\delta |s|}e^{-\beta(t-s)}.
\end{equation}
\end{definition}

\begin{remark}
We emphasize that in Definition 15 ii) in \cite{HMR}, it is imposed the condition $\delta\in[0,\beta)$, which is framed in the spectral perspective discussed in that paper (in addition to considering $J=\mathbb{R}_{0}^{+}$).     
\end{remark}

\begin{definition}
    The linear system \eqref{ALS} is said to be nonuniformly exponentially stable \textnormal{(NUES)} backward if there exist constant $\beta>0$, $\delta>0$  and $M\geq1$ such that for all $s\in\mathbb{R}$ and for all $t\leq s$:
\begin{equation}
\label{NUESMenosInfinito}
\|\Phi_{V}(t,s)\|\leq Me^{\delta |s|}e^{\beta(t-s)}.
\end{equation}
\end{definition}

From the concept of nonuniform exponential stability, we can extend the notion of uniform exponential detectability to this new framework.

\begin{definition}
    The system \eqref{control1a}-\eqref{control1b} is called nonuniformly exponentially detectable \textnormal{(NUED)} if there exists an output injection gain $L(t)$ such that the linear system \eqref{SistError} is \textnormal{NUES} forward.
\end{definition}

\begin{remark}
Given the nonuniform nature of the previous definition, it is reasonable to think that the output injection gain $L(t)$ is not necessarily bounded. 
\end{remark}

In view of the importance that adjoint and dual systems have played for a control system, given the linear system \eqref{ALS} we define its \textit{adjoint linear system} \begin{equation}
    \label{adjuntoV}
    \dot{x}(t)=-V^{T}(t)x(t)
\end{equation}
and its \textit{dual linear system} 
\begin{equation}
    \label{dualV}
    \dot{x}(t)=V^{T}(-t)x(t),
\end{equation} 
denoted by $\Psi_{a}(\cdot,\cdot)$ and $\Psi_{d}(\cdot,\cdot)$ as their transition matrices respectively. In this context, it is interesting to know what kind of stability properties each of these two systems possesses based on the stability of the system \eqref{ALS}, as stated in the following theorem.
\begin{theorem}
    \label{EstabilidadOriginalAdjunto} 
    The following three statements are equivalent:
    \begin{itemize}
        \item [i)] The original system \eqref{ALS} is \textnormal{NUES} forward.
        \item [ii)] The adjoint system \eqref{adjuntoV} is \textnormal{NUES} backward.
        \item [iii)] The dual system \eqref{dualV} is \textnormal{NUES} forward.
    \end{itemize}
\end{theorem}
\begin{proof} The proof will consist of verifying the equivalences $i)$ with $ii)$, and $i)$ with $iii)$.
\begin{itemize}
    \item $i)\Rightarrow ii)$ Let be  $\beta>0$, $\delta\geq0$ and $M\geq1$ constants such that verify
    $$\left\|\Phi_{V}(s,t)\right\|\leq Me^{-\beta(s-t)}e^{\delta|t|}$$
    for any $t\leq s$. Based on the previous estimate and by using the relation \eqref{Psiadjunto}, we have that for any $t\leq s$:
    $$\begin{array}{rcl}
\left \| \Psi_{a}(t,s)\right \|&=&\left\|\Phi_{V}^{T}(s,t)\right\|\leq Me^{-\beta(s-t)}e^{\delta|t|},\\
&\leq& Me^{-\beta(s-t)}e^{-\delta(s-t)}e^{\delta|s|},\\
&\leq&Me^{(\beta+\delta)(t-s)}e^{\delta|s|}.
\end{array}
$$
\item $ii)\Rightarrow i)$ Let be $\beta_{a}>0$, $\delta_{a}\geq0$ and $M_{a}\geq1$ constants such that verify
$$\left \| \Psi_{a}(t,s)\right \|\leq M_{a}e^{\beta_{a}(t-s)}e^{\mu_{a}|s|},$$
for any $t\leq s$. Then by considering the previous estimate and the relation \eqref{Psiadjunto}, we have that for any $t\leq s$:
$$\begin{array}{rcl}
\left \| \Phi_{V}(s,t)\right \|&=&\left\|\Psi_{a}^{T}(t,s)\right\|\leq M_{a}e^{\beta_{a}(t-s)}e^{\delta_{a}|s|},\\
&\leq& M_{a}e^{-\beta_{a}(s-t)}e^{-\delta_{a}(s-t)}e^{\delta_{a}|-t|},\\
&\leq&M_{a}e^{-(\beta_{a}+\delta_{a})(s-t)}e^{\delta_{a}|t|}.
\end{array}
$$

\item $i)\Rightarrow iii)$ Let be $\beta>0$, $\delta\geq0$ and $M\geq1$ constants such that verify \eqref{NUESInfinito} for any $t\geq s$ by replacing $t$ by $-s$ and $s$ by $-t$. Therefore, by using the previous estimate and \eqref{Psidual}, we have that for any $t\geq s$:
    $$\begin{array}{rcl}
\left \| \Psi_{d}(t,s)\right \|&=&\left\|\Phi_{V}^{T}(-s,-t)\right\|,\\
&\leq& Me^{-\beta(-s-(-t))}e^{\delta|-t|},\\
&\leq& Me^{-\beta(t-s)}e^{-\delta(t-s)}e^{\delta|-s|},\\
&\leq&Me^{-(\beta+\delta)(t-s)}e^{\delta|s|}.
\end{array}$$

\item $iii)\Rightarrow i)$ Let be $\beta_{d}>0$, $\delta_{d}\geq0$ and $M_{d}\geq1$ constants such that verify
$$\left \| \Psi_{d}(-s,-t)\right \|\leq M_{d}e^{-\beta_{d}(-s-(-t))}e^{\mu_{d}|-t|},$$
for any $t\geq s$. Then by considering the previous estimate and the relation \eqref{Psidual}, we have that for any $t\geq s$:
$$\begin{array}{rcl}
\left \| \Phi_{V}(t,s)\right \|&=&\left\|\Psi_{d}^{T}(-s,-t)\right\|,\\
&\leq& M_{d}e^{-\beta_{d}(-s-(-t))}e^{\delta_{d}|-t|},\\
&\leq& M_{d}e^{-\beta_{d}(t-s)}e^{-\delta_{d}(t-s)}e^{\delta_{d}|-s|},\\
&\leq&M_{d}e^{-(\beta_{d}+\delta_{d})(t-s)}e^{\delta_{d}|s|}.
\end{array}$$
\end{itemize}
\end{proof}
Based on the previous theorem, we emphasize the importance of relating the stability of the original system and its dual system, since this will allow connecting the concepts of observability and detectability with controllability and stabilizability in the nonuniform framework.

\section{Observability and detectability; main result}\label{sec-main-result}

Based on the theorem of the previous section, we can appreciate that the forward stability of the \eqref{ALS} system is equivalent to the same type of stability of system \eqref{dualV}. From this, it is useful to identify this relation in order to prove the main result of this section; nonuniform complete observability guarantees nonuniform exponential detectability, which is stated as follows: 
\begin{theorem}
\label{ObsImplicaDetect}
    If the system \eqref{control1a}-\eqref{control1b} is \textnormal{NUCO} on $\mathbb{R}$, then it is \textnormal{NUED}, that means, the system \eqref{SistError} is \textnormal{NUES} forward.  
\end{theorem}

The way to proceed to prove this theorem will be to focus on the results of section 4 of the work \cite{HMR}; namely, to generalize to the whole real line the different results that prove the close relationship between controllability, the matrix Riccati equations and stabilizability via feedback gain.

\subsection{Generalization of auxiliary results}
We emphasize that the results to be presented below have similar statements to those in section 4 of \cite{HMR}, except for the fact that in this work when studying the behavior on the whole real line, in the different estimates the term $|t|$ should appear instead of $t$.

Firstly, let us consider the family of Riccati differential equations parametrized by the number $\mathcal{L}>0$: 
\begin{equation}
\label{ER}
\begin{array}{c} 
\dot{S}(t)+\left(A(t)+\mathcal{L} I\right)^{T} S(t)+S(t)\left(A(t)+ \mathcal{L} I\right)\\
-S(t) B(t) B^{T}(t) S(t)=-I, 
\end{array}
\end{equation}
where $A(\cdot)$ and $B(\cdot)$ are the matrices of the control system \eqref{LTVControlabilidad}. The following result is the backbone of this subsection, being a generalization of Theorem 26 in \cite{HMR}.

\begin{theorem}
\label{T2}
Assume that the linear control system \eqref{LTVControlabilidad} is \textnormal{NUCC} on $\mathbb{R}$ and that the following additional properties are verified:
\begin{itemize}
\item[\textnormal{\textbf{(H1)}}] The corresponding plant has the property of nonuniform bounded growth on $\mathbb{R}$ with
constants $(K_{0},a,\eta)$,
\item[\textnormal{\textbf{(H2)}}] The property \eqref{alpha0alpha1enR} is verified with constants $\mu_{0}>0$ and $\mu_{1}>0$.
\end{itemize}

In addition, if $t\mapsto S_{\mathcal{L}}(t)$ is a solution of the Riccati equation \eqref{ER}
with $\mathcal{L}>2\theta_{1}$, where $\theta_{1}=\mu_{1}+4\eta$, then there exist a constant $M\geq 1$ and a feedback gain $F(t)$ such that the system \eqref{ALS} is \textnormal{NUES} forward, with $V(t)=A(t)-B(t)F(t)$, that is, for any $t\geq s$: 
\begin{equation}
\label{NUES}
\|\Phi_{A-BF}(t,s)\|\leq Me^{(\theta_{1}+\theta_{2})|s|} e^{-(\mathcal{L}-\theta_{1})(t-s)},
\end{equation}
where $\theta_{2}=\eta+2(\mu_{1}+\mu_{0})$. 
\end{theorem}


The proof of Theorem \ref{T2} essentially follows the same path that in \cite{HMR}. This result is extended to whole real line by considering the expression $|t|$ instead of $t$ in the nonuniform part of all estimates. The impact of this change does not generate substantial modifications in the proofs, except for some subtleties that must be taken into account when obtaining estimates of the different results.

In order to illustrate this fact, we will detail the proof of the following lemma.

\begin{lemma}
\label{evolutionoperatorresult}
If the linear control system \eqref{LTVControlabilidad} 
satisfies the hypothesis \textnormal{\textbf{(H1)}} 
then there exist functions $\zeta_1(\cdot)$ and $\zeta_2(\cdot)$ such that for any $\sigma>0$ we have that:
\begin{equation}
\label{integralPhi}
\zeta_1(\sigma)e^{-2\eta |t|}I\leq\displaystyle\int_{t}^{t+\sigma}\Phi_{A}^{T}(s,t)\Phi_{A}(s,t)\;ds\leq \zeta_2(\sigma)e^{2\eta |t|}I.
\end{equation}
\end{lemma}
\begin{proof}
By \textbf{(H1)} we know that the plant of \eqref{LTVControlabilidad} has the property of nonuniform bounded growth with 
constants $(K_0,a,\eta)$. Then, by using (\ref{Phibound}) we can write 
    $$
    \displaystyle
    \frac{1}{K_0^{2}}e^{-2a|s-t|-2\eta |s|}|x|^{2}\leq |\Phi_{A}(s,t
    )x|^{2}\leq K_0^{2}e^{2a|s-t|+2\eta |t|}|x|^{2},$$ 
for any $x\neq 0$, where the left estimation is obtained by noticing that $|x|=|\Phi_{A}^{T}(t,s)\Phi_{A}^{T}(s,t)x|$ combined
with (\ref{Phibound}). Now, we can deduce
$$\begin{array}{ll}
&\displaystyle\int_{t}^{t+\sigma}\frac{1}{K_0^{2}}e^{-2a|s-t|-2\eta |s|}|x|^{2}\;ds\leq\displaystyle\int_{t}^{t+\sigma}|\Phi_{A}^{T}(s,t)x|^{2}\;ds,\\
    &\leq\displaystyle\int_{t}^{t+\sigma} K_0^{2}e^{2a|s-t|+2\eta |t|}|x|^{2}\;ds.
    \end{array}
    $$

In relation to the left hand inequality, we can ensure that
$$
    \begin{array}{ll}
    &\displaystyle\int_{t}^{t+\sigma} \frac{1}{K_0^{2}}e^{-2a|s-t|-2\eta |s|}|x|^{2}\;ds\\
    &\geq\displaystyle\frac{1}{K_0^{2}}\displaystyle\int_{t}^{t+\sigma}e^{-2(a+\eta)(s-t)}e^{-2\eta|t|} |x|^2 ds,\\
    &=\displaystyle\frac{1}{K_0^{2}2(a+\eta)}\left[1-e^{-2(a+\eta)\sigma}\right]e^{-2\eta|t|}|x|^{2}
    \end{array}
    $$
and about the right expression we have that
$$
    \begin{array}{ll}
    &\displaystyle\int_{t}^{t+\sigma} K_0^{2}e^{2a|s-t|+2\eta |t|}|x|^{2}\;ds\\
    &\leq\displaystyle K_0^{2}\int_{t}^{t+\sigma}e^{2a(s-t)}e^{2\eta|t|}|x|^2 ds,\\
    &=\displaystyle \frac{K_0^{2}}{2a} \left[e^{2a\sigma}-1\right]e^{2\eta |t|}|x|^{2}
    \end{array}
    $$
and (\ref{integralPhi}) follows by considering $$\zeta_1(\sigma)=\displaystyle\frac{1}{K_0^{2}2(a+\eta)}\left[1-e^{-2(a+\eta)\sigma}\right]$$
and
$$\displaystyle\zeta_2(\sigma)= \frac{K_0^{2}}{2a} \left[e^{2a\sigma}-1\right].$$
\end{proof}

Considering the same techniques, the following results are obtained, which are direct extensions of those presented in section 4 of the paper \cite{HMR}.

\begin{proposition} 
\label{T3}
If the linear system \eqref{ALS}
verifies the following properties:

\begin{itemize}
\item[i)] There exists a positive definite operator $S(t)\in M_{n}(\mathbb{R})$ of class $C^1(\mathbb{R})$ and constants $C_1>0$, $C_2>0$, $\phi_1\geq0$ and $\phi_2\geq0$ such that
\begin{equation}
\label{CotaparaS}
C_1e^{-2\phi_1 |t|}I\leq S(t) \leq C_2 e^{2\phi_2 |t|}I, 
\end{equation}
\item[ii)] There exists a constant $\mathcal{L}>2\phi_1$ such that
\begin{equation}
\label{RiccatiS}
\dot{S}(t)+V^{T}(t) S(t)+S(t) V(t) \leq-(\mathrm{Id}+\mathcal{L} S(t)),
\end{equation}
\end{itemize}
then system \eqref{ALS} is \textnormal{NUES} forward, namely, there exist constants $M\geq 1$, $\lambda=\frac{\mathcal{L}}{2}-\phi_{1}>0$
and $\varepsilon=\phi_{1}+\phi_{2}\geq0$ such that
\begin{equation*}
\|\Phi_{V}(t,s)\|\leq  Me^{(\phi_1+\phi_2)|s|}e^{-(\frac{\mathcal{L}}{2}-\phi_1)(t-s)}  \quad \textnormal{for any $t\geq s$}.
\end{equation*}

\end{proposition}

\begin{lemma}
\label{L1}
If the linear control system \eqref{LTVControlabilidad} is \textnormal{NUCC} on $\mathbb{R}$, then the perturbed control system
\begin{equation}
\label{LTV-l}
\dot{z}(t)=[A(t)+\ell I]z(t)+B(t)u(t),
\end{equation}
with $\ell>0$, is also \textnormal{NUCC} on $\mathbb{R}$.
\end{lemma}

By considering the control system \eqref{LTV-l} with the particular choice $\ell=\frac{\mathcal{L}}{2}$ and the linear feedback input $u(t)=-F(t)z(t)$, then we have the linear system
\begin{equation}
\label{A+LI-BF}
\dot{z}(t)=\left[A(t)+\frac{\mathcal{L}}{2} I-B(t)F(t)\right]z(t).
\end{equation}



Subsequently, the following lemma establishes the stabilization via feedback of the system \eqref{LTV-l}.

\begin{lemma}
Under the assumptions of Theorem \textnormal{\ref{T2}}, for any $\mathcal{L}>2\theta_{1}$, with $\theta_{1}=\mu_{1}+4\eta$, then there exist $M\geq 1$
and a feedback gain $F(t)$ such that the shifted control system \eqref{A+LI-BF} is \textnormal{NUES}, namely, for all $t\geq s$:
\begin{equation}
\label{est-sof}
\|\Phi_{A+\frac{\mathcal{L}}{2}I-BF}(t,s)\|\leq Me^{(\frac{\mathcal{L}}{2}-\theta_1)(t-s)+(\theta_1+\theta_2) |s|},
\end{equation}
where $\theta_{2}=\eta+2(\mu_{1}+\mu_{0})$.
\end{lemma}

\medskip

Finally, by using the identity $$\Phi_{A-BF+\frac{\mathcal{L}}{2}I}(t,s)=\Phi_{A-BF}(t,s)e^{\frac{\mathcal{L}}{2}(t-s)}$$ 
combined with (\ref{est-sof}), it follows that
\begin{displaymath}
\|\Phi_{A-BF}(t,s)\|\leq Me^{(\theta_{1}+\theta_{2})|s|}e^{-(\mathcal{L}-\theta_1)(t-s)} \quad \textnormal{for all $t\geq s$},
\end{displaymath}
and the Theorem \ref{T2} follows.


\subsection{Proof of main result} By considering the extension of the stabilizability results and the close relationship between controllability and observability by means of the dual system, we present the proof of Theorem \ref{ObsImplicaDetect}. Figure \ref{fig-MainTheoProof} illustrates the scheme of the proof.

\begin{proof}
Using the Theorem \ref{dualNU}, we have that the nonuniform complete observability on $\mathbb{R}$ of the system \eqref{control1a}-\eqref{control1b} is equivalent to the nonuniform complete controllability on $\mathbb{R}$ of its dual system 
\begin{equation}
\label{dualAC}
    \dot{x}(t)=A^{T}(-t)x(t)+C^{T}(-t)u(t).
\end{equation}

Subsequently, based on Theorem \ref{T2} it follows that the nonuniform complete controllability of \eqref{dualAC} implies the stabilizability of the system, i.e., there exists a feedback gain $F(-t)$ such that the system
\begin{equation}
\label{estabilizabilidadAC}
\dot{x}(t)=(A^{T}(-t)-C^{T}(-t)F(-t))x(t)    
\end{equation}
is nonuniformly exponentially stable forward. In addition, by considering the Theorem \ref{EstabilidadOriginalAdjunto}, we can ensure that the system
\begin{equation}
    \dot{x}(t)=(A(t)-F^{T}(t)C(t))x(t)
\end{equation}
is nonuniformly exponentially stable forward. Finally, if we define $L(t)=F^{T}(t)$, we obtain that the system 
$$\dot{x}(t)=(A(t)-L(t)C(t))x(t)$$
is nonuniformly exponentially stable forward, which proves the nonuniform exponential detectability of the system \eqref{control1a}-\eqref{control1b}.
\end{proof}


\begin{remark}
It is important to highlight that based on the multiple behaviors existing in the nonautonomous context; bounded growth and stability, it makes sense to ask ourselves if it is possible to generalize the results obtained in this work to new contexts, either thinking about the close relationship between controllability and observability, or about the possibility of connecting stabilizability with detectability.     
\end{remark}


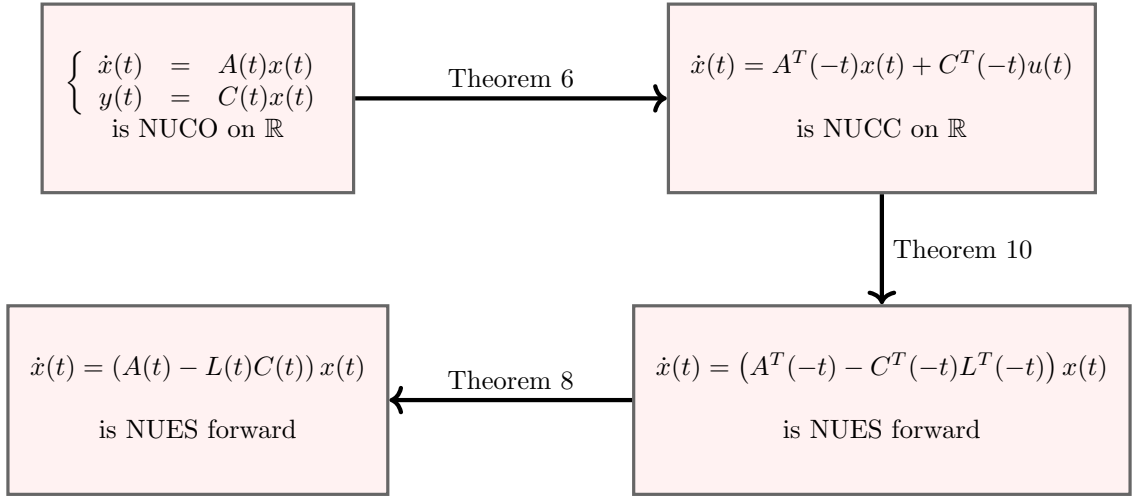
\begin{figure*}[h!]

\centering
\begin{tikzpicture}
[squarednode/.style={rectangle, draw=black!60, fill=red!5, very thick, minimum size=25mm},]
\node[squarednode]   at (0,0)   (11)                              {$\begin{array}{c}
\left\{\begin{array}{rcl}
\dot{x}(t)&=&A(t)x(t)\\
y(t)&=&C(t)x(t)
\end{array}\right.\\
\textnormal{is NUCO on } \mathbb{R}
\end{array}$
};

\node[squarednode]   at (9,0)    (12)        {$\begin{array}{c}
\dot{x}(t)=A^{T}(-t)x(t)+C^{T}(-t)u(t)\\\\
\textnormal{is NUCC on } \mathbb{R}  
\end{array}$
};

\node[squarednode]   at (0,-4)   (21)        {$\begin{array}{c}
\dot{x}(t)=\left(A(t) -L(t)C(t)\right)x(t)\\\\
\textnormal{is NUES forward}   
\end{array}
$
};

\node[squarednode]  at (9,-4)    (22)        {$\begin{array}{c}
\dot{x}(t)=\left(A^T(-t) - C^T(-t)L^T(-t)\right)x(t)\\\\
\textnormal{is NUES forward} 
\end{array}$
};

\draw[->,ultra thick] (11) -- (12)
node[midway, above] {Theorem \ref{dualNU}};

\draw[->,ultra thick] (12) -- (22)
node[midway, right] {Theorem \ref{T2}};

\draw[->,ultra thick] (22) -- (21)
node[midway, above] {Theorem \ref{EstabilidadOriginalAdjunto}};

\end{tikzpicture}

\caption{Graphical scheme of the proof of Theorem \ref{ObsImplicaDetect}.}
\label{fig-MainTheoProof}
\end{figure*}

\subsection{A counterexample}

This section will present an example of a planar system that will be nonuniformly exponentially detectable but will not verify the condition of nonuniform complete observability. In other words, the following example shows that Theorem \ref{ObsImplicaDetect} is not an equivalence. 

\begin{example}
    Consider the following control system: 
    \begin{subequations}
  \begin{empheq}[left=\empheqlbrace]{align}
    & \dot{x}(t)=A(t)x(t), 
     \label{Ejemplofinal1a} 
    \\
    & y(t)=C(t)x(t), 
     \label{Ejemplofinal1b}
  \end{empheq}
\end{subequations}
where 
\begin{equation*}
A(t)=
\begin{pmatrix}
-h(t) & 0\\
0 & h(t)
\end{pmatrix},
\quad  C(t)= \begin{pmatrix}
0 & 1
\end{pmatrix}
 \end{equation*}
and $h(t)=t^2$. To begin with, we will verify that \eqref{Ejemplofinal1a} does not satisfy the nonuniform Kalman condition, and based on Theorem \textnormal{\ref{p2l3obsNOUNIFORME}}, the system \eqref{Ejemplofinal1a}-\eqref{Ejemplofinal1b} is not nonuniformly completely observable. 

In fact, if we denote $\Phi_{A}(t,s)$ as the transition matrix of the system \eqref{Ejemplofinal1a}, then we will have that
$$\Phi_{A}(t,s)=\begin{pmatrix}
e^{\frac{s^3-t^3}{3}} & 0\\
0 & e^{\frac{t^3-s^3}{3}}
\end{pmatrix}.$$
We can conclude that $\|\Phi_{A}(t,s)\|=e^{\frac{|t^3-s^3|}{3}}$. Suppose that this system satisfies the nonuniform Kalman condition, i.e., there exist $\nu\geq0$ and $\alpha\in\mathcal{B}$ such that for any $t, s\in\mathbb{R}$:
$$e^{\frac{|t^3-s^3|}{3}}\leq  e^{\nu|s|}\alpha(|t-s|),$$
which is equivalent to
$$\begin{array}{rcl}\displaystyle\frac{|t-s|(t^2+ts+s^2)}{3}&\leq& \ln(\alpha(|t-s|))+\nu|s|,\\
|t-s|&\leq&  \displaystyle\frac{3\ln(\alpha(|t-s|))+3\nu|s|}{t^2+ts+s^2}.
\end{array}
$$

If we consider  $t=a_n$, $s=a_n-1$, where $a_n\to+\infty$ when $n\to+\infty$, we will have that the last inequality will be
$$1\leq \displaystyle+\frac{3\ln(\alpha(1))+3\nu|a_n-1|}{a_n^2+a_n(a_n-1)+(a_n-1)^2}$$
and when $n\to+\infty$ we obtain a contradiction. Similarly, if $t\leq s$ is considered, defining $t=a_{n}-1$ and $s=a_n$ leads again to a contradiction.


On the other hand, we now look for an observer $L(t)$ such that the system
\begin{equation}
\label{sistemaerrorejemplo}
\dot{e}(t)=(A(t)-L(t)C(t))e(t)
\end{equation}
is nonuniformly exponentially stable on $\mathbb{R}$. In fact, by choosing 
$$L(t)=\begin{pmatrix}
    0\\
    -2h(t)
\end{pmatrix},$$
we obtain that
$$A(t)-L(t)C(t)=\begin{pmatrix}
-h(t) & 0\\
0 & -h(t)
\end{pmatrix},$$
then we deduce that
$$\Phi_{A-LC}(t,s)=\begin{pmatrix}
e^{\frac{s^3-t^3}{3}} & 0\\
0 & e^{\frac{s^3-t^3}{3}}
\end{pmatrix}.$$

We want to prove that there exist $K\geq1$, $\alpha>0$ and $\mu\geq0$ such that for any $t\geq s$ it is verified that
$$\|\Phi_{A-LC}(t,s)\|\leq Ke^{-\alpha(t-s)+\mu|s|},$$
which is equivalent to prove that
$$-(t-s)(t^2+ts+s^2)\leq 3(\ln(K)-\alpha(t-s)+\mu|s|).$$

If we consider $K=e^3$, $\alpha=\mu=\frac{1}{3}$, the we will prove that for any $t>s$:
$$\begin{array}{rcl}
-(t-s)(t^2+ts+s^2)&\leq& 1-(t-s)+|s|,\\ t^2+ts+s^2&\geq&\displaystyle1-\frac{(|s|+1)}{t-s}.
\end{array}$$

In order to facilitate the writing process, for any $s\in\mathbb{R}$ fixed, we define the following functions:
$$f_{s}(t)=t^2+ts+s^2\quad \textnormal{and}\quad g_{s}(t)=1-\displaystyle\frac{(|s|+1)}{t-s}
$$
and then we will develop the proof, which will be divided into three cases.
\begin{itemize}
    \item \underline{Case 1}: If $s\in\mathbb{R}$ fixed and $1-\displaystyle\frac{3s^2}{4}\leq0$, then
    $$1\leq\frac{3s^2}{4}\leq \frac{3s^2}{4}+\left(t+\frac{s}{2}\right)^2=f_{s}(t),$$
and for all $t>s$ we have that $g_{s}(t)<1$, which allows us to conclude that for any $t>s$, $g_{s}(t)<f_{s}(t)$.

\item \underline{Case 2}: If $s<0$ and $1-\displaystyle\frac{3s^2}{4}>0$, that means $s\in\left(-\frac{2}{\sqrt{3}},0\right)$, then we obtain that 
$$g_{s}\left(-\frac{s}{2}\right)=\frac{s+2}{3s}<0<\displaystyle\frac{3s^2}{4}=f_{s}\left(-\frac{s}{2}\right).$$
In other words, by considering that the function $g_{s}(\cdot)$ is increasing on $\left(s,-\frac{s}{2}\right)$ and $f_{s}\left(\displaystyle-\frac{s}{2}\right)$ corresponds to the smallest image of $f_{s}$, then on the interval $\left(-\displaystyle\frac{s}{2},+\infty\right)$ we have that $f_{s}$ grows faster than $g_{s}$, then for any $t>s$, we have $g_{s}(t)<f_{s}(t)$.

\item \underline{Case 3}: If $s>0$ and $\displaystyle1-\frac{3s^2}{4}>0$, that means $s\in\left(0,\frac{2}{\sqrt{3}}\right)$, then the point $t_0=|s|+1+s=2s+1>s$ satisfies $g_{s}(t_0)=0$. After that, we note that $t_0>1$, then
$$f_{s}(t_0)=t_0^{2}+t_0s+s^2>1.$$
Besides, by considering $f_{s}$ is increasing on $\left(-\frac{s}{2},+\infty\right)$, $g_{s}$ is increasing on $(s,+\infty)$ and 
$t_0\in\left(-\frac{s}{2},+\infty\right)$, then for any $t\in\left(-\frac{s}{2},t_0\right]$ it is verified that
$$g_{s}(t)<0<f_{s}(t)$$
and if $t\in[t_0,+\infty)$, then
$$g_{s}(t)<1<f_{s}(t).$$
\end{itemize}

Based on the three cases analyzed, we conclude as requested, which allows us to conclude that the system \eqref{sistemaerrorejemplo} is nonuniformly exponentially detectable.
\end{example}





\begin{remark}
From the previous example and Theorem \textnormal{\ref{dualNU}} it is possible to construct a control system that does not verify the nonuniform complete controllability property but is nonuniformly completely stabilizable, proving that Theorem \textnormal{26} in \textnormal{\cite{HMR}} is not an equivalence in this context. 
\end{remark}


\section{Conclusions}

This paper has presented new results on observability and detectability of linear time-varying systems in a nonuniform setting. It put together classical results introduced by R. Kalman (controllability and observability) with more recent techniques for analyzing time-varying differential equations (the bounded growth property). It was shown that nonuniform complete observability defines a new class of time-varying linear systems that includes uniform complete observable time-varying linear systems but is strictly smaller than complete observable time-varying linear systems.
Furthermore, generalizing results from the work \cite{HMR}, which relate nonuniform controllability, Riccati matrix differential equations and nonuniform exponential stability, and proving the dual relationship between controllability and observability in the nonuniform framework, this work established that nonuniform complete observability guarantees nonuniform exponential detectability. Possible extensions of these ideas may include analysis on nonlinear perturbations and robustness of the nonuniform complete observability and nonuniform complete controllability properties.

\section*{Acknowledgement}

The first author has been partially supported by the grant FONDECYT Regular 1240361.

\bibliographystyle{plain}        
\bibliography{samplebib}

\begin{thebibliography}{10}

\bibitem{andersonmoore}
B.~D.~O. Anderson and J.~B. Moore.
\newblock Detectability and stabilizability of time-varying discrete-time
  linear systems.
\newblock {\em SIAM Journal on Control and Optimization}, 19(1):20--32, 1981.

\bibitem{Anderson}
B.D.O. Anderson, A.~Ilchmann, and F.~Wirth.
\newblock Stabilizability of linear time--varying systems.
\newblock {\em Systems \& Control Letters}, 62(9):747--755, 2013.
\newblock https://doi.org/10.1016/j.sysconle.2013.05.003.

\bibitem{Anderson69}
B.D.O. Anderson and J.~B. Moore.
\newblock New results in linear system stability.
\newblock {\em SIAM Journal on Control}, 7(3):398--414, 1969.
\newblock https://doi.org/10.1137/0307029.

\bibitem{Anderson67}
B.D.O. Anderson and L.M. Silverman.
\newblock Uniform complete controllability for time--varying systems.
\newblock {\em IEEE Transactions on Automatic Control}, 12(6):790--791, 1967.
\newblock 10.1109/TAC.1967.1098759.

\bibitem{BATISTA20173598}
P.~Batista, N.~Petit, C.~Silvestre, and P.~Oliveira.
\newblock Relaxed conditions for uniform complete observability and
  controllability of ltv systems with bounded realizations.
\newblock {\em IFAC-PapersOnLine}, 50(1):3598--3605, 2017.
\newblock 20th IFAC World Congress.

\bibitem{Callier}
F.~Callier and C.A. Desoer.
\newblock {\em Linear system theory}.
\newblock Springer-Verlag, New York, 1991.

\bibitem{ASP}
R.~Sepulchre D.~Aeyels and J.~Peuteman.
\newblock Asymptotic stability for time-variant systems and observability:
  Uniform and nonuniform criteria.
\newblock {\em Mathematics of Control, Signals and Systems}, 11(1):1--27, 1998.
\newblock https://doi.org/10.1007/BF02741883.

\bibitem{HMR}
I~Huerta, P.~Monzón, and G.~Robledo.
\newblock Controllability and feedback stabilizability in a nonuniform
  framework.
\newblock {\em Comptes Rendus. Mathématique}, 362:1667--1692, 2024.
\newblock https://doi.org/10.5802/crmath.672.

\bibitem{Ikeda2}
M.~Ikeda, H.~Maeda, and S.~Kodama.
\newblock Estimation and feedback in linear time-varying systems: A
  deterministic theory.
\newblock {\em SIAM Journal on Control}, 13(2):304--326, 1975.
\newblock https://doi.org/10.1137/0313018.

\bibitem{IlchmannNote}
A~Ilchmann.
\newblock {\em Contributions To Time-Varying Linear Control Systems}.
\newblock 2005.
\newblock \url{https://www.db-thueringen.de/receive/dbt_mods_00004754}.

\bibitem{Ioannou10.5555/211527}
P.~A. Ioannou and J.~Sun.
\newblock {\em Robust adaptive control}.
\newblock Prentice-Hall, Inc., USA, 1995.

\bibitem{Kalman}
R.~E. Kalman.
\newblock Contributions to the theory of optimal control.
\newblock {\em Bol. Soc. Mat. Mexicana}, 5:102--119, 1960.
\newblock https://api.semanticscholar.org/CorpusID:845554.

\bibitem{Kalman69}
R.~E. Kalman.
\newblock {\em Lectures on controllability and observability}.
\newblock Centro Internazionale Matematico Estivo (CIME), Edizioni Cremonese,
  Roma, 1969.

\bibitem{Kreindler}
E.~Kreindler and P.E. Sarachik.
\newblock On the concepts of controllaility and observability of linear
  systems.
\newblock {\em IEEE Transactions on Automatic Control}, 9(2):129--136, 1964.
\newblock 10.1109/TAC.1964.1105665.

\bibitem{Palmer}
K.J. Palmer.
\newblock Exponential dichotomy and expansivity.
\newblock {\em Annali di Matematica}, 185:S171--S185, 2006.
\newblock https://doi.org/10.1007/s10231-004-0141-5.

\bibitem{RAVI1992455}
R.~Ravi, A.M. Pascoal, and P.P. Khargonekar.
\newblock Normalized coprime factorizations for linear time-varying systems.
\newblock {\em Systems \& Control Letters}, 18(6):455--465, 1992.
\newblock https://doi.org/10.1016/0167-6911(92)90050-3.

\bibitem{sastry10.5555/63437}
S.~Sastry and M.~Bodson.
\newblock {\em Adaptive control: stability, convergence, and robustness}.
\newblock Prentice-Hall, Inc., USA, 1989.

\bibitem{SilvermanAnderson}
L.~Silverman and B.~Anderson.
\newblock Controllability, observability and stability of linear systems.
\newblock {\em SIAM Journal on Control}, 6(1):121--130, 1968.
\newblock https://doi.org/10.1137/0306010.

\bibitem{Silverman}
L.M. Silverman and B.D.O. Anderson.
\newblock Controllability, observability and stability of linear systems.
\newblock {\em SIAM Journal on Control}, 6(1):121--130, 1968.
\newblock https://doi.org/10.1137/0306010.

\bibitem{Sontag}
E.~Sontag.
\newblock {\em Mathematical control theory, deterministic finite dimensional
  systems}.
\newblock Springer, New York, 1998.

\bibitem{Tranninger}
M.~Tranninger, R.~Seeber, M.~Steinberger, and M.~Horn.
\newblock Uniform detectability of linear time varying systems with exponential
  dichotomy.
\newblock {\em IEEE Control Systems Letters}, 4(4):809--814, 2020.
\newblock 10.1109/LCSYS.2020.2992818.

\bibitem{tranninger2}
M.~Tranninger, R.~Seeber, S.~Zhuk, M.~Steinberger, and M.~Horn.
\newblock Detectability analysis and observer design for linear time varying
  systems.
\newblock {\em IEEE Control Systems Letters}, 4(2):331--336, 2020.
\newblock 10.1109/LCSYS.2019.2927549.

\bibitem{Zhou-21}
B.~Zhou.
\newblock Lyapunov differential equations and inequalities for stability and
  stabilization of linear time--varying systems.
\newblock {\em Automatica}, 131:109785, 2021.
\newblock https://doi.org/10.1016/j.automatica.2021.109785.

\end{thebibliography}



\appendix

\end{document}